\documentclass[%draft,%
	11pt,oneside,english,a4paper]{amsart}
\usepackage[T1]{fontenc}
\usepackage[utf8]{inputenc}
\usepackage[dvipsnames]{xcolor}
\usepackage[yyyymmdd,hhmmss]{datetime}
\usepackage[numbers,sort,compress]{natbib}
\usepackage{amstext,amsthm,amssymb}
\usepackage{xparse,mathrsfs,mathtools}
\usepackage[english]{babel}
\usepackage[pdftex,unicode=true,%
	pdfusetitle,pdfdisplaydoctitle=true,%
	pdfpagemode=UseOutlines,%
	bookmarks=true,bookmarksnumbered=false,bookmarksopen=true,bookmarksopenlevel=2,%
	breaklinks=true,%
	pdfencoding=unicode,psdextra,
	pdfcreator={},
	pdfborder={0 0 0},
	colorlinks=false
]{hyperref}
\usepackage{lmodern,microtype}
\usepackage[a4paper]{geometry}
\usepackage[capitalise,nameinlink]{cleveref}
\usepackage[cal=zapfc,bb=ams,frak=esstix,scr=boondoxo]{mathalfa}

\hypersetup{
	colorlinks=false,
	pdfborder={0 0 0},
	pdfborderstyle={/S/U/W 0},
}

\newtheorem*{thm*}{Theorem}
\newtheorem{thm}{Theorem}[section]
\newtheorem{lem}[thm]{Lemma}
\newtheorem{cor}[thm]{Corollary}
\newtheorem{prop}[thm]{Proposition}
\theoremstyle{remark}
\newtheorem{rem}[thm]{Remark}
\newtheorem*{rem*}{Remark}

\newtheorem*{example*}{Example}
\newtheorem*{examples*}{Examples}
\newtheorem{question}[thm]{Question}

\newcommand{\NN}{\mathbb{N}}
\newcommand{\ZZ}{\mathbb{Z}}
\newcommand{\QQ}{\mathbb{Q}}
\newcommand{\RR}{\mathbb{R}}
\newcommand{\CC}{\mathbb{C}}
\newcommand{\FF}{\mathbb{F}}
\newcommand{\integers}{\mathscr{O}}
\DeclareMathOperator{\Split}{Split}
\DeclareMathOperator{\Gal}{Gal}

\DeclarePairedDelimiter\parentheses{\lparen}{\rparen}
\DeclarePairedDelimiter\braces{\lbrace}{\rbrace}
\DeclarePairedDelimiter\abs{\lvert}{\rvert}
\NewDocumentCommand\set{ s o m o }{%
	\IfBooleanTF{#1}{\IfNoValueTF{#4}{\braces*{#3}}{\braces*{\,#3:#4\,}}}{%
	\IfNoValueTF{#2}{\IfNoValueTF{#4}{\braces{#3}}{\braces{\,#3:#4\,}}}{%
	\IfNoValueTF{#4}{\braces[#2]{#3}}{\braces[#2]{\,#3:#4\,}}}}%
}

\crefname{section}{§}{§§}
\crefformat{equation}{#2(#1)#3}
\crefformat{enumi}{#2(#1)#3}
\numberwithin{equation}{section}
\makeatletter
\DeclareRobustCommand{\thmprime}{\begingroup%
	\expandafter\in@\expandafter b\expandafter{\f@series}%
	\ifin@ \boldmath \fi%
	\(\m@th{}{\raisebox{3pt}{\(\scriptstyle\prime\)}}\)%
\endgroup}
\makeatother
\NewDocumentEnvironment{manual}{O{thm}m}{%
	\addtocounter{thm}{-1}%
	\renewcommand\thethm{#2}\begin{#1}%
}{\end{#1}}

\usepackage{xifthen,ifthen}
\makeatletter
\newcounter{@ToDo}
\newcommand{\todo@helper}[1]{%
	({\color{blue}TODO~\arabic{@ToDo}: {#1\@addpunct{.}}})%
}
\newcommand{\todo}[1]{\stepcounter{@ToDo}%
	\relax\ifmmode\text{\todo@helper{#1}}%
	\else\todo@helper{#1}\fi%
}
\makeatother

\title{On polynomials with roots modulo almost all primes}
\date{\today{}}

\makeatletter\@namedef{subjclassname@2020}{\textup{2020} Mathematics Subject Classification}\makeatother
\subjclass[2020]{% See http://www.ams.org/msc/msc2020.html
	Primary
	11R09; % Polynomials (irreducibility, etc.)
	Secondary
	11R32. % Galois theory
}

\author{Christian~Elsholtz}

\author{Benjamin~Klahn}

\author{Marc~Technau}

\address{
	Christian~Elsholtz \and Benjamin~Klahn \and Marc~Technau\\%
	Institut für Analysis und Zahlentheorie\\%
	TU~Graz\\%
	Kopernikusgasse~24/II\\%
	8010~Graz\\%
	Austria}
\email{%
	elsholtz@math.tugraz.at\textnormal{, }%
	klahn@math.tugraz.at\textnormal{, }%
	mtechnau@math.tugraz.at%
}

\begin{document}
\begin{abstract}
    Call a monic integer polynomial \emph{exceptional} if it has a root modulo all but a finite number of primes, but does not have an integer root.
    We classify all irreducible monic integer polynomials $h$ for which there is an irreducible monic \emph{quadratic} $g$ such that the product $gh$ is exceptional.
    We construct exceptional polynomials with all factors of the form $X^{p}-b$, $p$ prime and $b$ square free.
\end{abstract}
\maketitle

% -------------------------------------------------------
\section{Introduction}

\subsection{Earlier work on intersective polynomials}
The use of local methods in algebraic number theory and arithmetic geometry is well established.
For instance, a standard approach towards investigating the solubility of a given Diophantine equation in some number field proceeds by first studying solubility in all completions.
In certain good situations, local solutions can be `patched together' to yield global solubility.
The availability of such local--global principles (or `Hasse principles') or the lack/failure thereof is a major topic of on-going research in the area.

The object of the present investigation has a somewhat different flavour.
We start by pointing towards some related items in the literature.
In~\cite{Bilu}, motivated by earlier investigations regarding finiteness results for the number of solutions of the Diophantine equation $f(x) = n!$ (with $f\in\ZZ[X]$ a polynomial of degree exceeding~one and $n$ some fixed positive integer), Berend and Bilu studied the following problem:
\begin{question}\label{q:BerendBilu}
	Given a polynomial $f\in\ZZ[X]$, decide whether or not for \emph{every} integer~$m$ its reduction $(f\bmod m) \in (\ZZ/m\ZZ)[X]$ admits a root.
\end{question}
Equivalently, one may ask if $f$ admits a solution in the $p$-adic integers for every rational prime~$p$.
By Hensel's lemma, the latter essentially boils down to $f\bmod p$ admitting a root and $f\bmod\Delta$ admitting a root, where $\Delta$ is some integer related to the primes $p$ for which $f\bmod p$ fails to be separable; this integer may be computed effectively from the discriminants of the irreducible factors of $f$ (see~\cite{Bilu} for the details).

If $f\in\ZZ[X]$ has a root in the integers, then $f$ trivially satisfies the property expounded in \cref{q:BerendBilu}.
On the other hand, the converse does not hold in general (see, e.g., \cite[Example~1]{Bilu}).
%However, if a monic polynomial $f$ has degree at most~$4$ and has a root modulo all but a finite number of primes, then $f$ does in fact have a root in the integers.
However, if a monic polynomial $f$ has degree at most~$4$ and has a root modulo all primes, then $f$ does in fact have a root in the integers~\cite[Remark~2]{Bilu}.
Polynomials $f\in\ZZ[X]$ without a root in the integers, yet satisfying the property in \cref{q:BerendBilu} have been called \emph{intersective} in the literature.

\medskip

We mention some Galois-theoretic results on such polynomials, noting that the connection to Galois theory shall prove fruitful in~\cref{sec:CharacterisingExceptionality} below.
In~\cite{Sonn1}, Sonn showed that any finite non-cyclic solvable group can be realized as the Galois group of an intersective polynomial.
Furthermore, in~\cite{Sonn}, Sonn showed that, for non-cyclic, non-solvable groups, being realizable as a Galois group is equivalent to being realizable as the Galois group of an intersective polynomial.
In~\cite{Koenig}, König constructed non-solvable groups that can be realized as the Galois group of an intersective polynomial.

Finally, we point out that intersective polynomials have appeared in the literature with regard to other questions whose proofs are of a rather more analytic nature, involving tools from Fourier analysis and ergodic theory.
We mention only a select few results.
For instance, Kamae and Mendes France~\cite{Kamae} have shown that the polynomials $f\in\ZZ[X]$ for which the value set $\set{ \abs{f(n)} }[ n\in\ZZ ] \setminus\set{0}$ is a van~der Corput set are precisely the intersective ones; here a subset $\mathscr{H}$ of positive integers is called a \emph{van~der Corput set} if for any sequence $(u_n)_n$ on $\RR/\ZZ$ uniform distribution of every differenced sequence $(u_{n+h}-u_n)_n$, $h\in\mathscr{H}$, implies uniform distribution of $(u_n)_n$.
In additive number theory, a well-known result of Furstenberg--S\'{a}rk\H{o}zy~\cite{Furstenberg,Sarkozy} asserts that a subset $A$ of positive integers such that the difference set $A-A = \set{ a-a' }[ a,a'\in A ]$ contains no non-zero squares must be small in a suitable quantitative sense.
This generalises to sets $A\subseteq\NN$ avoiding value sets of intersective polynomials.
The interested reader is referred to~\cite{Rice}.
In a similar direction, Bergelson, Leibman and Lesigne~\cite{Bergelson} have obtained a generalisation of van der Waerden's well-known theorem in Ramsey theory, involving finite families of `jointly intersective' integer-valued polynomials with rational coefficients.

\subsection{Exceptional polynomials}
For our investigation we relax the condition of having a root modulo every $m$ slightly.
The following may serve as an initial guiding question:
\begin{manual}[question]{\ref{q:BerendBilu}\thmprime}\label{q:BerendBilu'}
	Given a polynomial $f\in\ZZ[X]$, decide whether or not $f\bmod p$ admits a root for all but at most finitely many primes~$p$.
\end{manual}
We remark here that, as already pointed out in~\cite{Bilu}, the algorithmic aspect of such questions is (at least in principle) settled by work of J.\ Ax.
It transpires that \cref{q:BerendBilu} is equivalent to \cref{q:BerendBilu'} and checking whether $f\bmod\Delta$ admits a root, where $\Delta$ is the integer mentioned above.
For simplicity, we shall restrict our attention to \emph{monic} polynomials $f$ in the sequel.
Let $\mathscr{P}(f)$ denote the set of rational primes $p$ for which $f\bmod p$ admits a root.
It is elementary to show that $\mathscr{P}(f)$ is an infinite set and the density theorem of Chebotarëv even shows that $\mathscr{P}(f)$ has positive natural density in the set of primes (see \cite[Theorem~2]{Bilu} or \cref{thm:BerendBilu} below).
Clearly, if $f$ admits a root in the integers, then $f\bmod m$ trivially has a solution modulo every integer~$m$.
Hence, we shall call a monic polynomial $f\in\ZZ[X]$ \emph{exceptional} if $f$ has no root in the integers and $\mathscr{P}(f)$ contains all but at most finitely many primes. For instance, both of the following polynomials are exceptional:
\begin{equation}\label{examplePolys}
	g = (X^2 - 2) (X^2 - 3) (X^2 - 6)
		\quad\text{and}\quad
	h = (X^2 + 108) (X^3 + 2).
\end{equation}
Incidentally, neither of these two polynomials is intersective as can be seen, for instance, upon reducing them modulo $2^6$.
That $g$ is exceptional can be verified via \cref{exceptionalPower} below.
That $h$ is exceptional is a folklore exercise in Galois theory using that $-108$ is the discriminant of $X^3 + 2$, thus identifying $X^2 + 108$ as the quadratic resolvent of $X^3 + 2$.

\subsection{Plan of the paper}

In this paper we do the following.
\begin{enumerate}
	\item We adapt a Galois-theoretic answer of \cite{Bilu} to \cref{q:BerendBilu} to our \cref{q:BerendBilu'} (\cref{sec:CharacterisingExceptionality}).
	\item%\label{enum:QuadraticClassification}
	We outline how the above can be used to construct exceptional polynomials with few irreducible factors (\cref{sec:ConstructingExceptionalPolynomials:ViaGaloisTheory}).
	\item We completely classify all irreducible monic polynomials $h\in\ZZ[X]$ for which there exists an irreducible monic quadratic polynomial $g\in\ZZ[X]$ such that their product $gh$ is exceptional (\cref{sec:CharacterisationWithQuadraticFactor}).
	\item We construct exceptional polynomials as products of factors $X^p - b$ (\cref{sec:BuildingWithCyclotomicFactors}).
\end{enumerate}
Most of the proofs are postponed until~\cref{sec:Proofs}.

\subsection{Acknowledgements}

The authors would like to thank Yuri Bilu, Joachim König and Marc Munsch for sharing some helpful comments on an earlier draft of this manuscript.
Further thanks are due to the anonymous referee for valuable feedback.

The first-named author was supported by a joint FWF--ANR project \emph{`ArithRand'} (FWF~I 4945-N and ANR-20-CE91-0006).
The first- and second-named authors acknowledge support of the \emph{Austrian Science Fund (FWF)}, project number~W1230.

% -------------------------------------------------------
\section{Characterising exceptionality}
\label{sec:CharacterisingExceptionality}

The aim of this section is to give a Galois-theoretic answer \cref{q:BerendBilu'}.
Most of what we do in this section is not entirely new and can be found (implicitly or explicitly) in~\cite{Bilu}.

Fix a monic polynomial $f\in\ZZ[X]$ of degree $n$.
As $f$ is exceptional if and only if the product of the prime divisors of $f$ is exceptional, we may suppose throughout the rest of the paper that $f$ is square free.
Let $L = \Split(f, \QQ) \subset \CC$ denote the splitting field of $f$ (embedded into~$\CC$) and $G = \Gal(L/\QQ)$ its Galois group.
Fixing an enumeration of the roots $\alpha_1,\alpha_2,\dots,\alpha_n$ of $f$, the faithful action of $G$ on $\set{ \alpha_1, \alpha_2, \dots, \alpha_n }$ induces an embedding
\(
    G \hookrightarrow S_n
\)
by means of which we shall regard $G$ as a subgroup of $S_n$ in the sequel.
Moreover, for $i=1,2,\dots, n$, let
\begin{equation}\label{eq:Hi:Def}
	H_i = \Gal(L/\QQ(\alpha_i)) \leqslant G.
\end{equation}

In the following we briefly review some well-known results from algebraic number theory for the reader's convenience.
Proofs may be found in standard texts such as~\cite{Neukirch} or~\cite{Narkiewicz}.
For a rational prime $p$ we say that $f$ has \emph{factorisation pattern} $n_1 \leq n_2 \leq \ldots \leq n_t$ modulo~$p$ if the reduction $f \bmod p$ factors as
\[
	(f \bmod p) = f_1 f_2 \cdots f_{t},
\]
where the $f_i$ are irreducible polynomials in $\FF_p[X]$ and the $n_i$ denote their respective degrees.
Suppose now that $p$ is a rational prime, unramified in~$L$, with corresponding factorisation
\[
	p \mkern 1mu \integers_L = \mathfrak{p}_1 \mathfrak{p}_2 \cdots \mathfrak{p}_{k}
\]
into prime ideals of the ring $\integers_L$ of algebraic integers of $L$.
Then, for every $i=1,\ldots,k$, there exists a unique element $\sigma_{\mathfrak{p}_i} \in G$ such that for every $x\in \integers_L$ we have
\[
	\sigma_{\mathfrak{p}_i}(x) \equiv x^p \mod{ \mathfrak{p}_i}.
\]
The automorphism $\sigma_{\mathfrak{p}_i}$ is called the \emph{Frobenius element at $\mathfrak{p}_i$} and it has the property that, when considered as an element of $S_n$, its cycle type equals the factorisation pattern of $f \bmod p$.
The set
\(
	C_p = \set{ \sigma_{\mathfrak{p}_i} }[
		\mathfrak{p}_i ~ \text{lying above} ~ p \mkern 1mu \integers_L
	]
\)
constitutes a conjugacy class of $G$.
Chebotarëv's density theorem determines the natural density of primes $p$ for which a given conjugacy class $C$ of $G$ occurs as $C_p$:
\begin{thm*}[Chebotarëv's density theorem over $\QQ$]%\label{thm:Chebotarev}
	Let $K/\QQ$ be a Galois extension and let $C$ be a conjugacy class of $G=\Gal(K/\QQ)$.
	Then the natural density of rational primes $p$ for which $C_p = C$ exists, and equals $\abs{C}/\abs{G}$.
\end{thm*}

This was applied by Berend and Bilu~\cite{Bilu} to obtain the following result:
\begin{thm}\label{thm:BerendBilu}
	Keep the notation from above.
	Then the natural density $\delta(\mathscr{P}(f))$ of rational primes $p$ for which $f\bmod p$ has a root is
	\[
		\delta(\mathscr{P}(f))
		\coloneqq \lim_{N\to\infty} \frac{\abs{\set{ \text{primes } p\leq N }[ f\bmod p \text{~has a root} ]}}{\abs{\set{\text{primes } p\leq N}}}
%		= \frac{ \abs{\bigcup_{i=1}^n H_i} }{ \abs{G} }.
		= \frac{1}{\abs{G}} \abs[\bigg]{\bigcup_{i=1}^n H_i},
	\]
	where the $H_i$ are given by~\cref{eq:Hi:Def}.
\end{thm}

Recall that only finitely many rational primes~$p$ ramify in~$L$.
Moreover, since, for unramified primes $p$, the Frobenius elements over $p$ have the same cycle type as the factorisation pattern of $f \bmod p$, we see that $f$ is exceptional if and only if it has no root in the integers and every element of $G$ has a fixed point.
The second part of the last statement is clearly equivalent to every element of $G$ fixing a root $\alpha_i$ of $f$.
We obtain the following result, which is a variant of a result obtained by Sonn~\cite{Sonn1} concerning intersective polynomials.
\begin{prop}\label{coverExcep}
	Keep the notation from above.
	Then $f$ is exceptional if and only if $f$ has no integer root and
	\begin{equation}\label{eq:GaloisGroupAsUnion}
		G = \bigcup_{i=1}^n H_i.
	\end{equation}
\end{prop} 

Note that, in particular, \cref{coverExcep} furnishes a means to decide whether some polynomial $f$ is exceptional in terms of its Galois group.
From this one easily deduces the next result which is also due to Sonn~\cite{Sonn1} in the case of intersective polynomials (see also~\cite{Brandl,Koenig}).
\begin{cor}\label{cor:ExceptionalPolynomialsAreReducible}
	All exceptional polynomials are reducible.
\end{cor}
\begin{proof}
	Suppose that $f$ is exceptional and irreducible. Then the action of the Galois group of $f$ acts transitively on the roots of $f$. However, a consequence of Burnside's Lemma says that not all elements of a transitive group have a fixed point unless the group is trivial. Thus, by Theorem \cref{coverExcep} the Galois group of $f$ must be trivial, contradicting that $f$ does not have an integer root.
\end{proof}

It seems worth pointing out that \cref{cor:ExceptionalPolynomialsAreReducible} may be viewed as a `local to global'-type result: an exceptional polynomial has a linear factor modulo almost all primes $p$ (which, incidentally, lifts to a linear factor over the $p$-adic integers in all but at most finitely many cases) and is, consequently, reducible.
The conclusion is then that the exceptional polynomial is already reducible over $\QQ$.
Note that our definition of exceptionality does preclude the existence of a linear factor over $\QQ$, though.

\begin{rem}
	For a group $G$ we let $s(G)$ be the minimal number of proper subgroups of $G$ having the property that the union of their conjugates cover $G$ and intersect trivially.
	Further, we let $r(G)$ be the minimal number of irreducible factors of an intersective polynomial with Galois group equal to $G$ over $\QQ$.
	Then \cref{coverExcep} implies that $s(G) \leq r(G)$.
	In~\cite{SonnBubboloni}, Bubboloni and Sonn showed that for $G = S_n$ one has $r(G) = s(G)$ or $r(G) = s(G)+1$ for any $n$.
	Further it was shown that $r(G) = s(G)$ for odd $n$ and for some even values of $n$.
	For $G = S_n$ or $G = A_n$ the following is known:
	\begin{itemize}
	    \item if $G = S_n$ and $s(G) = 2$, then $3 \leq n \leq 6$;
	    \item if $G = A_n$ and $s(G) = 2$, then $4 \leq n \leq 8$.
	\end{itemize}
	Rabayev and Sonn~\cite{SonnRabayev} showed that in any of the above cases $r(G) = 2$ by constructing explicit irreducible polynomials whose product has Galois group $G$.
\end{rem}

% -------------------------------------------------------
\section{Constructing exceptional polynomials}
%\label{sec:ConstructingExceptionalPolynomials}

% -------------------------------------------------------
\subsection{Constructing exceptional polynomials via Galois theory}
\label{sec:ConstructingExceptionalPolynomials:ViaGaloisTheory}

As a first application of \cref{coverExcep}, we show how any given Galois extension $L/\QQ$ with non-cyclic Galois group $G$ may be used to construct an exceptional polynomial.
Indeed, since $G$ is non-cyclic, every element of $G$ is contained in a proper subgroup of $G$.
Thus, we may cover $G$ with proper non-trivial subgroups:
\[
	G = \bigcup_{i=1}^{n} H_i.
\]
By Galois correspondence, $H_i$ may be written as $H_i = \Gal(L/K_i)$ for some intermediate field $K_i$ of the extension $L/\QQ$.
The primitive element theorem guarantees that $K_i$ is simple, that is, $K_i = \QQ(\alpha_i)$ for some $\alpha_i \in K_i$.
We let $f_i$ be the minimal polynomial of $\alpha_i$ over $\QQ$ and denote by $\tilde{K}_i$ the splitting field of $f_i$ in $L$.
Then the splitting field $M$ of $f = f_1 f_2 \cdots f_n$ is 
\[
	M = \tilde{K}_1 \tilde{K}_2 \cdots \tilde{K}_n \subseteq L.
\]
Consider the projection
\[
	\Gal(L/\QQ) \twoheadrightarrow \Gal(M/\QQ), \quad
	\psi \mapsto \psi\rvert_{M}.
\]
As every element of $G$ fixes one of the $\alpha_i$, we see that any element of $\Gal(f,\QQ)$ fixes a root of $f$, showing that $f$ is exceptional.
\begin{rem}
	As $G$ is assumed non-cyclic, one may choose a denser cover $(H_{i})_{I}$ of subgroups of $G$ such that the intersection of the conjugates of the $H_{i}$ is trivial.
	In that way, the splitting field of $f$ actually becomes equal to $L$ such that the Galois group of $f$ is equal to $G$.
	This shows that any non-cyclic group can be realised as the Galois group of an exceptional polynomial.
	Applying stronger tools, Sonn used such a cover in~\cite{Sonn1} to prove that any solvable non-cyclic group can be realised as the Galois group of an intersective polynomial.
\end{rem}

% -------------------------------------------------------
\subsection{Exceptional polynomials with quadratic factor}
\label{sec:CharacterisationWithQuadraticFactor}

In view of~\cref{cor:ExceptionalPolynomialsAreReducible}, one may try to construct exceptional polynomials with \emph{few} irreducible factors.
The guiding principle for the present section may be phrased as follows:
\begin{question}\label{q:ForceExceptionality}
	Given an irreducible monic polynomial $h\in\ZZ[X]$ without a root in the integers.
	Does there exist some irreducible polynomial $g\in\ZZ[X]$ such that $f = gh$ is exceptional?
\end{question}
We are able to settle this question under the additional restriction that $g$ be \emph{quadratic}.
Indeed, we have the following full characterisation:
\begin{thm}\label{quadPairTheo}
	Let $h$ be an irreducible monic integer polynomial and let $G$ denote its Galois group.
	Then the following statements are equivalent:
	\begin{enumerate}
		\item\label{enum:QuadraticFactorCreatesExceptionality}
		There is an irreducible monic quadratic polynomial $g\in\ZZ[X]$ such that the product $gh$ has a root modulo all but finitely many primes.
		\item\label{enum:SpecialSubgroupExistence}
		There is a subgroup $H$ of $G$ of index two such that every element in the coset $G \setminus H$ has a \emph{unique} fixed point, i.e., fixing a unique root of $h$.
	\end{enumerate}
\end{thm}

\begin{rem}\label{rem:Completing}
	It can be observed that `completing towards exceptionality' as asked for in \cref{q:ForceExceptionality} can easily be extended to `completing towards intersectivity'.
	Indeed, starting with an exceptional polynomial $f$ it is easy to use quadratic reciprocity and the Chinese remainder theorem to construct from the finitely many rational primes not belonging to $\mathscr{P}(f)$ an integer $d\equiv 1 \bmod 4$ such that $(X^2-d)f$ is intersective.
\end{rem}

The next result shows that a polynomial $h$ can only satisfy one of the equivalent conditions in \cref{quadPairTheo} if the degree of $h$ is odd.
\begin{prop}\label{nIsOdd}
	Let $h$ be an irreducible monic integer polynomial of degree $n$.
	If~\cref{enum:QuadraticFactorCreatesExceptionality} in \cref{quadPairTheo} holds, then $n$ is odd.
\end{prop}
Next, we want to show that for all odd integers $n$, there is actually an irreducible monic integer polynomial $h$ of degree $n$ that satisfies the conditions of \cref{quadPairTheo}.
More precisely, we show that, for any integer $n$, the Dihedral group $D_n$ with $2n$ elements is realisable as the Galois group of an irreducible polynomial \emph{of degree $n$} and that, when $n$ is odd, all elements outside the unique subgroup of index two fix a root of that polynomial.
\begin{prop}\label{dihedralProp}
	Let $n \geq 3$ be a positive integer.
	There is an irreducible polynomial $h$ of degree $n$ such that $G\coloneqq \Gal(h, \QQ) \cong D_n$.
	Let $H$ be the unique subgroup of $G$ of index~$2$.
	If $n$ is odd, then every element of $G \setminus H$ fixes a root of $h$.
\end{prop}
Combining \cref{quadPairTheo} and \cref{dihedralProp} we easily see that, for every odd integer $n \geq 3$, there is an irreducible monic integer polynomial $h$ of degree $n$ such that there is an irreducible monic quadratic polynomial $g$ for which $gh$ is exceptional.

\begin{cor}\label{dihedralSol}
	For any $n\geq 2$ there exists a monic irreducible non-exceptional polynomial $h\in\ZZ[X]$ of degree $n$ with the following property: $h$ can be completed to an exceptional polynomial $gh$ by multiplying by a quadratic polynomial $g$ if and only if $n$ is odd.
\end{cor}

% -------------------------------------------------------
\subsection{Exceptional polynomials built from \texorpdfstring{$X^p-b$}{Xp-b} factors}
\label{sec:BuildingWithCyclotomicFactors}

We turn our attention to considering exceptional polynomials where all factors are of the form $X^p - b$ for a fixed prime $p$ and a square-free integer $b$.
By Eisenstein's irreducibility criterion, these polynomials are irreducible in $\QQ[X]$.

For a finite set $\mathscr{L}$ of rational primes, consider the set
\[
	\mathscr{B}_{\mathscr{L}} = \set[\Big]{ \prod_{\ell \in \mathscr{P}} \ell }[
		\emptyset \neq \mathscr{P} \subseteq \mathscr{L},\,
		\mathscr{P} = \mathscr{L}\cap[\min\mathscr{P},\max\mathscr{P}]
	]
\]
consisting of all square-free integers $>1$ whose set of prime factors is a subset of consecutive elements of $\mathscr{L}$.
Finally, let $f_{\mathscr{L}}$ be the polynomial given by
\[
    f_{\mathscr{L}} = \prod_{b \in \mathscr{B}_{\mathscr{L}}} (X^p - b).
\]
\begin{thm}\label{exceptionalPower}
	Let $\mathscr{L}$ be a set of primes, and let the notation be as above.
	Then the polynomial $f_{\mathscr{L}}$ is exceptional if and only if $\abs{\mathscr{L}} \geq p$.
\end{thm}

\begin{examples*}%\leavevmode
	Given any three primes $q<r<s$, the following polynomials are exceptional:
	\begin{enumerate}
		\item \(\displaystyle
			\prod\nolimits_b (X^2 - b)
		\), where $b$ ranges over $\set{ q, r, qr }$;
		\item \(
			g = (X^2 - 2) (X^2 - 3) (X^2 - 6)
		\)
		from~\cref{examplePolys};
		\item \(\displaystyle
			\prod\nolimits_b (X^3 - b)
		\), where $b$ ranges over $\set{ q, r, s, qr, qs, rs, qrs }$;
		\item \(
			(X^3 - 2) (X^3 - 3) (X^3 - 5) (X^3 - 6) (X^3 - 15) (X^3 - 30)
%			= \prod_{ b\in\set{ 2, 3, 5, 6, 15, 30 } } (X^3 - b)
		\).
	\end{enumerate}
\end{examples*}

% -------------------------------------------------------
\section{Proofs}
\label{sec:Proofs}

% -------------------------------------------------------
\subsection{Proofs of the results about quadratic factors}
Let $g$ and $h$ be polynomials with rational coefficients of degree $m$ and $n$, respectively, and let $f\coloneqq gh$.
Let $M\coloneqq \Split(f,\QQ) \subset \CC$, $K\coloneqq \Split(g,\QQ) \subset M$, and $L = \Split(h,\QQ) \subset M$ denote the splitting fields (in $\CC$) of $f$, $g$ and $h$, respectively.
Finally, let $N\coloneqq \Gal(M/\QQ)$, $P=\Gal(K/\QQ)$ and $G=\Gal(L/\QQ)$ be the Galois group of $f$, $g$ and $h$, respectively.

We have an embedding
\[
    \iota\colon N \hookrightarrow P \times G, \quad
    \psi \mapsto (\psi\rvert_{K}, \psi\rvert_{L})
\]
and the image of $N$ under $\iota$ consists of all pairs $(\psi_1,\psi_2) \in P \times G$ where $\psi_1\rvert_{K\cap L} = \psi_2\rvert_{K\cap L}$.
It follows, in particular, that $\iota$ is an isomorphism when $K \cap L = \QQ$.
\begin{proof}[Proof of \cref{quadPairTheo}]
	We keep the notation from the above discussion but assume now that $g$ has degree~$2$.
	Assume first that $h$ is an irreducible monic integer polynomial such that there is an irreducible monic quadratic $g$ such that $f\coloneqq gh$ is exceptional.
	
	We claim that $K \subset L$.
	If not, then $K \cap L = \QQ$ and $\iota$ is therefore an isomorphism.
	As both $g$ and $h$ are irreducible there are elements $\psi_1 \in P$ and $\psi_2 \in G$ such that $\psi_1$ does not fix a root of $g$ and $\psi_2$ does not fix a root of $h$.
	It follows that $\psi\coloneqq \iota^{-1}((\psi_1,\psi_2)) \in G$ does not fix a root of $f$, by \cref{coverExcep} contradicting that $f$ is exceptional.
	
	Thus, $K \subset L$ and $H\coloneqq \Gal(L/K)$ is a subgroup of $G$ of index~$2$.
	The preimage $\iota^{-1}(P \times H)$ corresponds exactly to the elements in $N = \Gal(M/\QQ)$ that fix a root of $g$, and, by \cref{coverExcep}, every element of $G \setminus H$ must therefore fix at least one root of $h$.
	
	As $h$ is irreducible, we get via an enumeration of the roots of $h$ a transitive action of $G$ on $\mathscr{X}\coloneqq \set{ 1,2,\dots , n }$.
	By Burnside's lemma we thereby find that
	\begin{align*}
		\abs{G} &
		= \abs{G} \mkern 1mu \abs{\mathscr{X} / G}
		= \sum_{g \in G} \abs{\mathscr{X}^{g}} %\\ &
		= \sum_{g \in G\setminus H} \abs{\mathscr{X}^{g}} + \sum_{g \in H} \abs{\mathscr{X}^{g}} \\ &
		\geq \abs{G \setminus H} + \abs{H} \mkern 1mu \abs{\mathscr{X} / H}
		\geq \parentheses[\big]{ \tfrac{1}{2} + \tfrac{1}{2} } \abs{G}
		= \abs{G},
	\end{align*}
	where the first inequality only can be an \emph{equality} when every element of $G \setminus H$ fixes a \emph{unique} root of $h$.
	
	For the converse, let $H \leqslant G$ be such a subgroup and let $\QQ \subset K \subset L$ be the corresponding intermediate field extension.
	We may write $K = \QQ(\sqrt{d})$ for some square-free integer $d$.
	Let 
	\[
		g\coloneqq X^2 - d.
	\]
	The Galois group $N$ of $gh$ under $\iota$ may then be described as the pairs of maps $(\psi_1,\psi_2) \in P \times G$ such that $\psi_1(\sqrt{d}) = \psi_2(\sqrt{d})$.
	For $\psi \in N$, write $\iota(\psi) \coloneqq (\psi_1,\psi_2)$.
	If $\psi(\sqrt{d}) = \sqrt{d}$, then $\psi_1$ fixes a root of $g$, and if $\psi(\sqrt{d})=-\sqrt{d}$, then $\psi_2 \in G \setminus H$.
	Therefore, by construction, $\psi_1$ fixes a root of $h$.
	Thus, all elements in $N$ have a fixed point, and, by \cref{coverExcep}, $gh$ has a root modulo all but finitely many primes, as desired.
\end{proof}
\begin{proof}[Proof of \cref{nIsOdd}]
	Let $h$ be an irreducible monic integer polynomial of degree $n$ such that there exists a quadratic, $g$, such that $gh$ is exceptional.
	Moreover, assume for the sake of a contradiction that $n$ is even.
	
	Let $G\coloneqq \Gal(L/\QQ)$ be the Galois group of $h$ over $\QQ$ where $L$ is the splitting field of $h$, and let $H \leqslant G$ be a subgroup of $G$ satisfying~\cref{enum:SpecialSubgroupExistence} in \cref{quadPairTheo}.
	By Galois theory, we must have $H = \Gal(L/\QQ(\sqrt{d}))$ for some square-free integer $d$.
	
	Consider any map $\psi \in G \setminus H$, which must satisfy $\psi(\sqrt{d}) = -\sqrt{d}$ and therefore have even order, say, $\operatorname{ord}(\psi) = 2^{k} m$ for some positive $k$ and some odd integer $m$.
	Then $\psi^m$ has order $2^{k}$ and the cycles of $\psi^m$ in $S_n$ must consequently all have length $2^{\ell}$ for some $\ell$.
	Furthermore, $\psi^m$ is in $G \setminus H$ and must therefore have a \emph{unique} cycle of length~$1$.
	However, as $n$ is even, and the cycles of $\psi^m$ correspond to a partition of $n$, there must be an even number of $1$-cycles in $\psi^m$, contradicting~\cref{enum:SpecialSubgroupExistence} in \cref{quadPairTheo}.
\end{proof}

\begin{proof}[Proof of \cref{dihedralProp}]
	We first show that for every $n \geq 3$ the Dihedral group $D_n$ with $2n$ elements can be realized as the Galois group of an irreducible polynomial of degree $n$.
	Let $g \in \QQ[X]$ be a polynomial with Galois group $G$ isomorphic to $D_n$.
	The existence of $g$ is ensured by Shafarevich's theorem, as $D_n$ is solvable for every $n \geq 3$.
	Let $L = \Split(g, \QQ) \subset \CC$ be the splitting field of $g$.
	Notice that $[L :\QQ] = 2n$ and $G = \Gal(L/\QQ)$.
	Let $K \subset L$ be an intermediate field such that $H \coloneqq  \Gal(L/K) \leqslant G$ is a subgroup of $G$ of order~$2$, meaning that $[K:\QQ] = n$.
	Let $\alpha$ be an integer in $L$ that is a primitive element of $K$, i.e., $K = \QQ(\alpha)$.
	Let $h = \operatorname{Irr}(\alpha, \QQ) \in \ZZ[X]$ be the minimal polynomial of $\alpha$ over $\QQ$, which is a polynomial of degree $n$.
	Let $G' = \Gal(h, \QQ)$ and $M = \Split(h, \QQ)$.
	We claim that $G' \cong D_n$.
	As $L$ is normal and $\alpha \in L$ we have $K \subset M \subset L$.
	It is easily seen that $H$ is not normal in $G$ and that the inclusion $K \subset M$ therefore is strict, implying that $M = L$.
	This shows that $D_n$ can be realized as the Galois group of an irreducible polynomial of degree $n$.
	
	Let $n \geq 3$ be an odd integer and let $h$ be a polynomial of degree $n$ with Galois group $G\coloneqq \Gal(h,\QQ) \cong D_n$.
	$G$ has a unique subgroup $H$ of index~$2$ and that subgroup is cyclic of order $n$.
	Furthermore, all elements in $G \setminus H$ have order~$2$.
	
	Fixing an enumeration of the roots of $h$ we again find an embedding
	\[
		\eta\colon G \hookrightarrow S_n
	\]
	via the obvious action of $G$ on the roots of $h$.
	Every element in $\eta(G \setminus H)$ has order~$2$, and hence must consist of $2$-cycles and $1$-cycles.
	However, as $n$ is odd, there must be at least one $1$-cycle in every element of $\eta(G \setminus H)$, meaning that every element of $G \setminus H$ fixes a root of $h$, as claimed.
\end{proof}

\begin{proof}[Proof of \cref{dihedralSol}]
	\cref{nIsOdd} implies that $h$ necessarily has odd degree.
	Conversely, if $n$ is odd, then the result follows immediately upon taking $h$ to be any irreducible polynomial of degree $n$ with Galois group isomorphic to $D_n$.
	By \cref{cor:ExceptionalPolynomialsAreReducible} such a polynomial is guaranteed to be non-exceptional and the result follows from \cref{dihedralProp}.
\end{proof}

The following observation was kindly pointed out to the second author by Joachim König (private communication).
\begin{rem}
	The results of this subsection can be generalized in the following way: let $p$ and $q$ be primes with $p \equiv 1 \bmod{q}$.
	Up to isomorphism, there is a unique non-trivial semi-direct product $G\coloneqq C_p \rtimes C_{q}$ of the cyclic groups of order $p$ and $q$.
	In a similar fashion as in the proof of \cref{dihedralProp}, one might show that $G$ can be realized as the Galois group of an irreducible polynomial $h$ of degree $p$ with splitting field, say, $L$.
	Choose a subfield $K \subset L$ with $[L:K] = p$.
	It follows that $\QQ \subset K$ is normal and thus, that $K$ is the splitting field of a polynomial $g$ of degree $q$.
	Putting $H = \Gal(L/K) \leqslant G$, it follows that $H$ is a subgroup of $G$ of order $p$, and, as in the proof of \cref{dihedralProp}, it can be seen that every element of $G\setminus H$ fixes a root of $h$ and therefore that $gh$ is exceptional.
	We conclude that, for any prime $q$, there is a prime $p$ such that there are irreducible monic integer polynomials $g$ and $h$ of degree $q$ and $p$, respectively, such that the product $gh$ is exceptional. 
\end{rem}

% -------------------------------------------------------
\subsection{Proof of \texorpdfstring{\cref{exceptionalPower}}{Theorem\autoref{exceptionalPower}}}
We keep the notation from the formulation of \cref{exceptionalPower}.
For the proof of \cref{exceptionalPower}, we shall check that every element of the Galois group of $f_{\mathscr{L}}$ fixes a root of $f_{\mathscr{L}}$.
If $b$ is some non-negative real number, we write $b^{1/p}$ for the (unique) positive real root of $X^p - b$.
The other roots of $X^p - b$ are then given by $\zeta_p^{\nu} b^{1/p}$ ($\nu=1,\ldots,p-1$), where $\zeta_p$ may be taken to be any primitive $p$-th root of unity.
It follows that the splitting field $L_{\mathscr{L}}$ of $f_{\mathscr{L}}$ is
\[
	L_{\mathscr{L}} = \QQ(\set{ \ell^{1/p} }[ \ell\in\mathscr{L} ] \cup \set{\zeta_p}).
\]
In~\cite[Theorem~2]{Besicovitch} it was shown that the $p$-th roots of different primes are independent, meaning that
\[
	[\QQ(\set{ \ell^{1/p} }[ \ell\in\mathscr{L} ]) : \QQ] = p^n.
\]
As
\(
	[\QQ(\zeta_p), \QQ] = p-1
\),
it follows that
\[
	[L_{\mathscr{L}}:\QQ] = p^n (p-1).
\]
Every element of $\psi \in G \coloneqq \Gal(L_{\mathscr{L}}/\QQ)$ is completely determined by its values on $\ell^{1/p}$ (for $\ell\in\mathscr{L}$) and $\zeta_p$.
Moreover, for every $\ell\in\mathscr{L}$ there is some $\nu(\ell)\in\set{0,\ldots,p-1}$ such that
\(
	\psi(\ell^{1/p}) = \zeta_p^{\nu(\ell)} \ell^{1/p}
\).
Similarly, $\psi(\zeta_p) = \nu(\zeta_p)$ for some $\nu(\ell)\in\set{1,\ldots,p-1}$.
Because of reasons of cardinality, it follows that mapping $\psi$ to the function $\nu$ just described gives rise to a bijection of sets
\begin{equation}\label{eq:GaloisGroupBijection}
	G
	\xrightarrow{~ 1:1 ~}
	\set{\text{maps } \nu\colon\mathscr{L}\cup\set{\zeta_p}\to\set{0,\ldots,p-1} \text{~such that~} \nu(\zeta_p) \neq 0 }.
\end{equation}

\medskip

We will use the following well-known elementary fact from zero-sum theory.
We include a proof for the reader's convenience.
\begin{lem}\label{zeroSum}
	Let $(g_i)_i$ be a finite sequence of not necessarily distinct elements in some abelian group $(G,+)$ with $n$ elements.
	If the sequence consists of at least $n$ terms, then it admits a subsequence of consecutive elements whose sum is zero in $G$.
\end{lem}
\begin{proof}
	It suffices to prove the result for a sequence $g_1,\ldots,g_n$ of length~$n$.
	Let $s_k = \sum_{i=1}^k g_i$.
	Then either $\set{s_1,\ldots,s_n} = G \ni 0$ or we have $s_k = s_\ell$ for two integers $1\leq k < \ell \leq n$.
	In the former case, we are already done, but in the latter case, we are also done, for $0 = s_\ell - s_k = \sum_{i=k+1}^\ell g_i$.
\end{proof}

\begin{proof}[Proof of \cref{exceptionalPower}]
	For $p=2$ the result is immediate from the multiplicativity of the Legendre symbol.
	Therefore, we may assume that $p$ is an \emph{odd} prime.
	
	We first show that, for any set $\mathscr{L}$ of at least $p$ rational primes, the polynomial $f_{\mathscr{L}}$ is exceptional.
	To this end, let $\psi$ be an element of $G$ and let $\nu$ be the map corresponding to $\psi$ under~\cref{eq:GaloisGroupBijection}.
	By \cref{zeroSum}, there is a subset $\mathscr{L}' \subseteq \mathscr{L}$ of consecutive primes such that
	\[
		1 = \prod\nolimits_{\ell} \zeta_p^{\nu(\ell)},
	\]
	where the product is taken over all $\ell\in\mathscr{L}'$.
	This implies that
	\[
		\psi\parentheses[\Big]{ \prod\nolimits_{\ell} \ell^{1/p} }
		= \prod\nolimits_{\ell} \psi(\ell^{1/p})
		= \prod\nolimits_{\ell} \zeta_p^{\nu(\ell)} \ell^{1/p}
		= \parentheses[\Big]{ \prod\nolimits_{\ell} \zeta_p^{\nu(\ell)} } \prod\nolimits_{\ell} \ell^{1/p}
		= \prod\nolimits_{\ell} \ell^{1/p}.
	\]
	Hence, $\psi$ fixes a root of $f_{\mathscr{L}}$ and we deduce~\cref{eq:GaloisGroupAsUnion}.
	\cref{coverExcep} now shows that $f_{\mathscr{L}}$ is exceptional.
	
	Now, let $\mathscr{L}$ be a set consisting of exactly $p-1$ primes.
	Consider the map $\psi\in G$ which fixes $\zeta_p$ and maps $\ell^{1/p}$ to $\zeta_p\ell^{1/p}$ for every $\ell\in\mathscr{L}$.
	Pick any root $\alpha$ of $f_{\mathscr{L}}$.
	Then, for some subset $\set{\ell_1,\ell_2, \dots, \ell_{r}} \subseteq \mathscr{L}$ and some $\nu_0$ we have $\alpha = \zeta_p^{\nu_0} (\ell_1 \cdots \ell_r)^{1/p}$.
	Therefore, $\alpha$ gets mapped by $\psi$ to $\zeta_p^r \alpha \neq \alpha$, implying that $\psi$ does not fix any root of $f_{\mathscr{L}}$ and, hence, that $f_{\mathscr{L}}$ is not exceptional.
\end{proof}
\begin{rem}
	In \cref{exceptionalPower} we have restricted ourselves to considering a very special class of polynomials.
	It is an interesting question to determine the minimum size of an arbitrary set $\mathscr{B}$ of square-free integers such that the polynomial
	\[
		f_{\mathscr{B}}(x) = \prod_{b\in\mathscr{B}} (X^{p}-b)
	\]
	is exceptional.
	\cref{exceptionalPower} yields such a set $\mathscr{B}$ with $\abs{\mathscr{B}} = \frac{1}{2} p(p+1)$.
	For $p=2$ and $p=3$ one can check that this is actually the minimum size, and we conjecture that this is indeed always the minimum.
	
	For an integer $n$, let $\mathscr{S}_n \subseteq \mathbb{F}_p[X_1,X_2,\ldots,X_n]$ denote the set
	\[
		\mathscr{S}_{n} \coloneqq \set[\Big]{ \sum_{X\in\mathscr{X}} X }[
			\emptyset \neq \mathscr{X} \subseteq \set{X_1,\ldots,X_n}
		].
	\]
	Call a subset $\mathscr{T}_{n} \subseteq \mathscr{S}_{n}$ \emph{good} if the polynomial
	\(
		\prod_{t\in\mathscr{T}_{n}} t \in \mathbb{F}_p[X_1,X_2,\ldots,X_n]
	\)
	vanishes on all points of $\mathbb{F}_{p}^{n}$.
	Arguing as in the proof of \cref{exceptionalPower}, it follows that the minimal size of a set of square-free integers $\mathscr{B}$ such that $f_{\mathscr{B}}$ is exceptional is
	\[
		\min_{n\in\mathbb{N}} \min\set{
			\abs{\mathscr{T}_{n}}
		}[
			\text{good subsets~} \mathscr{T}_{n}\subseteq\mathscr{S}_{n}
		].
	\]
	This translates the question of determining the minimum size of $\mathscr{B}$ into problem of combinatorial/geometric flavour.
\end{rem}

%\clearpage\pagecolor{black!5}
% -------------------------------------------------------
%\phantom.\vfill
%\bibliographystyle{00-plainnat}%
%\bibliography{galois-project}%

\begin{thebibliography}{15}
	\bibitem[{Berend} and {Bilu}(1996)]{Bilu}
	D.~{Berend} and {\relax Yu}.~{Bilu}.
	\newblock {Polynomials with roots modulo every integer}.
	\newblock \emph{{Proc.\ Am.\ Math.\ Soc.}}, 124\penalty0 (6):\penalty0 1663--1671, 1996.

	\bibitem[{Bergelson} et~al.(2008){Bergelson}, {Leibman}, and {Lesigne}]{Bergelson}
	V.~{Bergelson}, A.~{Leibman}, and E.~{Lesigne}.
	\newblock {Intersective polynomials and the polynomial {S}zemer\'{e}di theorem}.
	\newblock \emph{{Adv.\ Math.}}, 219\penalty0 (1):\penalty0 369--388, 2008.

	\bibitem[{Besicovitch}(1940)]{Besicovitch}
	A.~S. {Besicovitch}.
	\newblock {On the linear independence of fractional powers of integers}.
	\newblock \emph{{J.\ Lond.\ Math.\ Soc.}}, 15:\penalty0 3--6, 1940.

	\bibitem[{Brandl} et~al.(2001){Brandl}, {Bubboloni}, and {Hupp}]{Brandl}
	R.~{Brandl}, D.~{Bubboloni}, and I.~{Hupp}.
	\newblock {Polynomials with roots mod \(p\) for all primes \(p\)}.
	\newblock \emph{{J.\ Group Theory}}, 4\penalty0 (2):\penalty0 233--239, 2001.

	\bibitem[{Bubboloni} and {Sonn}(2016)]{SonnBubboloni}
	D.~{Bubboloni} and J.~{Sonn}.
	\newblock {Intersective \(S_n\) polynomials with few irreducible factors}.
	\newblock \emph{{Manuscr.\ Math.}}, 151\penalty0 (3-4):\penalty0 477--492, 2016.

	\bibitem[{Furstenberg}(1977)]{Furstenberg}
	H.~{Furstenberg}.
	\newblock {Ergodic behavior of diagonal measures and a theorem of {S}zemer\'{e}di on arithmetic progressions}.
	\newblock \emph{{J.\ Analyse Math.}}, 31:\penalty0 204--256, 1977.

	\bibitem[{Kamae} and {\relax Mend\`es France}(1978)]{Kamae}
	T.~{Kamae} and M.~{\relax Mend\`es France}.
	\newblock {Van~der Corput's difference theorem}.
	\newblock \emph{{Israel J.\ Math.}}, 31\penalty0 (3-4):\penalty0 335--342, 1978.

	\bibitem[{K\"onig}(2018)]{Koenig}
	J.~{K\"onig}.
	\newblock {On intersective polynomials with non-solvable Galois group}.
	\newblock \emph{{Commun.\ Algebra}}, 46\penalty0 (6):\penalty0 2405--2416, 2018.

	\bibitem[{Narkiewicz}(2004)]{Narkiewicz}
	W.~{Narkiewicz}.
	\newblock \emph{{Elementary and analytic theory of algebraic numbers}}, volume 322 of \emph{Springer Monogr.\ Math.}
	\newblock Berlin: Springer, 3rd edition, 2004.

	\bibitem[{Neukirch}(1999)]{Neukirch}
	J.~{Neukirch}.
	\newblock \emph{{Algebraic number theory. Transl.\ from the German by Norbert Schappacher}}, volume 322 of \emph{Grundlehren Math.\ Wiss.}
	\newblock Berlin: Springer, 1999.

	\bibitem[{Rabayev} and {Sonn}(2014)]{SonnRabayev}
	D.~{Rabayev} and J.~{Sonn}.
	\newblock {On Galois realizations of the \(2\)-coverable symmetric and alternating groups}.
	\newblock \emph{{Commun.\ Algebra}}, 42\penalty0 (1):\penalty0 253--258, 2014.

	\bibitem[{Rice}(2019)]{Rice}
	A.~{Rice}.
	\newblock {A maximal extension of the best-known bounds for the Furstenberg--S\'{a}rk\"{o}zy theorem}.
	\newblock \emph{{Acta Arith.}}, 187\penalty0 (1):\penalty0 1--41, 2019.

	\bibitem[{S\'{a}rk\H{o}zy}(1978)]{Sarkozy}
	A.~{S\'{a}rk\H{o}zy}.
	\newblock {On difference sets of sequences of integers. {I}}.
	\newblock \emph{{Acta Math.\ Acad.\ Sci.\ Hungar.}}, 31\penalty0 (1-2):\penalty0 125--149, 1978.

	\bibitem[{Sonn}(2008)]{Sonn1}
	J.~{Sonn}.
	\newblock {Polynomials with roots in \(\mathbb{Q}_p\) for all \(p\)}.
	\newblock \emph{{Proc.\ Am.\ Math.\ Soc.}}, 136\penalty0 (6):\penalty0 1955--1960, 2008.

	\bibitem[{Sonn}(2009)]{Sonn}
	J.~{Sonn}.
	\newblock {Two remarks on the inverse Galois problem for intersective polynomials}.
	\newblock \emph{{J.\ Th\'eor.\ Nombres Bordx.}}, 21\penalty0 (2):\penalty0 437--439, 2009.
\end{thebibliography}

\vfill%
\end{document}